\documentclass[11pt]{amsart}
\usepackage{amsfonts}
\usepackage{amsmath}
\usepackage{amsthm}
\usepackage{url}
\usepackage[all]{xy}
\usepackage{latexsym}
\usepackage{amssymb}
\usepackage[cp850]{inputenc}
\usepackage{amsfonts}
\usepackage{graphicx}

\usepackage[abs]{overpic}

\newtheorem{sat}{Theorem}[section]		
\newtheorem{lem}[sat]{Lemma}
			
\newtheorem{prop}[sat]{Proposition}

\newtheorem*{defi*}{Definition}			
\newtheorem*{bei*}{Example}
\newtheorem*{sat*}{Theorem}				
\newtheorem*{kor*}{Corollary}
\newtheorem*{rmk*}{Remark}				
	
\newtheorem*{quest*}{Question}	
\newtheorem{fact}{Fact}	

\let\ssection=\section
\renewcommand{\section}{\setcounter{equation}{0}\ssection}

\newtheorem*{namedtheorem}{\theoremname}
\newcommand{\theoremname}{testing}
\newenvironment{named}[1]{\renewcommand{\theoremname}{#1}\begin{namedtheorem}}{\end{namedtheorem}}

\theoremstyle{remark}

\newtheorem*{namedtheoremr}{\theoremnamer}
\newcommand{\theoremnamer}{testing}

			\newcommand{\BH}{\mathbb H}
\newcommand{\BR}{\mathbb R}			
			
\newcommand{\BS}{\mathbb S}			\newcommand{\BZ}{\mathbb Z}

\newcommand{\BX}{\mathbb X}

\newcommand{\CC}{\mathcal C}		
		\newcommand{\CF}{\mathcal F}
\newcommand{\CG}{\mathcal G}		\newcommand{\CH}{\mathcal H}
		
		\newcommand{\CL}{\mathcal L}
\newcommand{\CM}{\mathcal M}		
		
		\newcommand{\CR}{\mathcal R}
\newcommand{\CS}{\mathcal S}

\newcommand{\D}{\partial}
\newcommand{\DD}{\nabla}

\DeclareMathOperator{\PSL}{PSL}		
\DeclareMathOperator{\vol}{vol}		

\DeclareMathOperator{\arccosh}{arccosh}

\newcommand{\comment}[1]{}

\DeclareMathOperator{\syst}{syst}
\DeclareMathOperator{\supp}{supp}
\DeclareMathOperator{\I}{i}

\begin{document}

\title[]{Variations on a theorem of Birman and Series}
\author{Anna Lenzhen}
\address{IRMAR, Universit\'e de Rennes 1}
\email{anna.lenzhen@univ-rennes1.fr}
\author{Juan Souto}
\address{IRMAR, Universit\'e de Rennes 1}
\email{juan.souto@univ-rennes1.fr}
\date{\today}


\begin{abstract}
Suppose that $\Sigma$ is a hyperbolic surface and $f:\BR_+\to\BR_+$ a monotonic function. We study the closure in the projective tangent bundle $PT\Sigma$ of the set of all geodesics $\gamma$ satisfying $\I(\gamma,\gamma)\le f(\ell_\Sigma(\gamma))$. For instance we prove that if $f$ is unbounded and sublinear then this set has Hausdorff dimension strictly bounded between $1$ and $3$.
\end{abstract}

\maketitle

\section{Introduction}

Let $\Sigma$ be a hyperbolic surface which, unless said explictly, we suppose closed. It is by now a classical result of Birman and Series \cite{Birman-Series} that the set of all geodesics in $\Sigma$ with boundedly many self-intersections has Hausdorff dimension $1$. Recently, Sapir \cite{Sapir} proved that this remains true for the set of all geodesics $\gamma:\BR\to P$ in a pair of pants $P$ with infinitely many self-intersections, as long as these self-intersections are themselves sparsely distributed along the geodesic in the sense that the number of self-intersections within $\gamma[-L,L]$ grows subquadratically when $L\to\infty$. In this note we consider a different variation of the Birman-Series theorem: while the set considered by Sapir is not closed, we will be interested in the closure of the set of closed geodesics whose self-intersection number is bounded in terms of the length. 

More precisely, whenever $f:\BR_+\to\BR_+$ is a positive monotonic function, we will be interested in the closure of the set $\CS_f$ of those (non-constant) primitive periodic geodesics $\gamma$ in $\Sigma$ satisfying
\begin{equation}\label{eq-flegal}
\I(\gamma,\gamma)\le f(\ell_\Sigma(\gamma)).
\end{equation}
Here $\I(\gamma,\gamma)$ is the self-intersection number of $\gamma$ and $\ell_\Sigma(\gamma)$ is the length of $\gamma$ with respect to the hyperbolic metric. A periodic geodesic satisfying \eqref{eq-flegal} will be called {\em $f$-simple}. Moreover, we will identify geodesics with the associated geodesic lifts to the projectivized tangent bundle $PT\Sigma$. Finally, given that small changes on the function $f$ may result in individual curves becoming $f$-simple or no longer being so, we will consider the set
$$X_f=\bigcap_{T=0}^\infty\overline{\{\gamma\in\CS_f\text{ with }\ell_\Sigma(\gamma)\ge T\}}\subset PT\Sigma$$
instead of working directly with the closure $\overline\CS_f$ of the set of $f$-simple geodesics. Note that $X_f$ might be equivalently described as the set of points in $PT\Sigma$ which belong to some Hausdorff limit of a sequence of pairwise distinct $f$-simple geodesics. In particular, $\overline\CS_f\setminus X_f$ consists of countably many closed curves, which implies that both sets $\overline\CS_f$ and $X_f$ have identical Hausdorff dimension.

As we will see, the set $X_f$ depends only very coarsely on the function $f$ and the concrete hyperbolic metric on $\Sigma$. Our first observation is that if $f$ is superlinear then $X_f$ is equal to the whole of $PT\Sigma$:

\begin{sat}\label{sat1}
If $f:\BR_+\to\BR_+$ is such that $\limsup_{t\to\infty}\frac{f(t)}t=\infty$, then $X_f=PT\Sigma$ for every closed hyperbolic surface $\Sigma$.
\end{sat}

In particular, for any superlinear function $f$ we have that $X_f$ has Hausdorff dimension $\dim_{\CH}X_f=3$. On the other hand, if the function $f$ is bounded then $X_f$ is contained in the set of all geodesics with boundedly many self-intersections and thus has Hausdorff dimension $1$ by the Birman-Series theorem \cite{Birman-Series}. Our main results concern what happens if $f$ is unbounded and sublinear. First we prove that for any two functions with these properties we get identical sets:

\begin{sat}\label{sat2}
If  $\Sigma$ is a closed hyperbolic surface, then there is a set $\BX(\Sigma)\subset PT\Sigma$ with $X_f=\BX(\Sigma)$ whenever $f:\BR_+\to\BR_+$ is unbounded and satisfies $\limsup_{t\to\infty}\frac{f(t)}t=0$.
\end{sat}

Recall that the geodesic flows of any two closed hyperbolic surfaces $\Sigma,\Sigma'$ of the same genus are orbit equivalent, meaning that there is a H\"older homeomorphism
$$\Phi:PT\Sigma\to PT\Sigma'$$ 
which maps geodesics to geodesics. It follows either directly from Theorem \ref{sat2} or from the description of $\BX(\Sigma)$ given in \eqref{eq-setX} that $\Phi$ maps $\BX(\Sigma)$ homeomorphically to $\BX(\Sigma')$. In particular, since the H\"older constant of $\Phi$ is close to $1$ if the surfaces $\Sigma$ and $\Sigma'$ are close in the Hausdorff topology, it follows that for each $g$ the function 
\begin{equation}\label{eq-continuity}
\CM_g\to[1,\infty),\ \ \Sigma\mapsto\dim_\CH\BX(\Sigma)
\end{equation}
which associates to each surface $\Sigma$ the Hausdorff dimension of the set $\BX(\Sigma)$ is continuous on the moduli space $\CM_g$ of closed hyperbolic surfaces of genus $g$. Our next aim is to show that it is not constant and to estimate it in terms of the geometry of $\Sigma$:

\begin{sat}\label{prop-H-bound}
There are constants $C>0$ and $c>0$ such that
$$3-c\cdot\frac{\syst(\Sigma)}{\vol(\Sigma)}\le \dim_{\CH}\BX(\Sigma) \le3-C\cdot\left(\frac{\syst(\Sigma)}{\vol(\Sigma)}\right)^2$$
for every closed hyperbolic surface $\Sigma$. Here $\syst(\Sigma)$ is the systole of $\Sigma$, i.e. the length of the shortest closed geodesic.
\end{sat}

In the case that the surface $\Sigma$ satisfies $10\syst(\Sigma)\le\vol(\Sigma)$, which happens for instance whenever the genus is at least $6$, one can give concrete numerical values to the constants in Theorem \ref{prop-H-bound}. In fact we get in this case that
$$2+\sqrt{1-10\frac{\syst(\Sigma)}{\vol(\Sigma)}}\le \dim_\CH\BX \le 2+\sqrt{1-\frac{4\cdot\syst(\Sigma)^2}{\vol(\Sigma)^2+4\cdot\syst(\Sigma)^2}}.$$
In a different direction, recall that it is due to Buser and Sarnak \cite{Buser-Sarnak} that the maximal value of the systole function on $\CM_g$ is of the order of $\frac{\log(g)}g$. In particular, we get from Theorem \ref{prop-H-bound} that the image of \eqref{eq-continuity} is an interval of the form $[m_g,3)$ where, for large $g$, we have
$$ 3-B\frac{\log(g)}{g}\le m_g\le 3-A\frac{\log(g)^2}{g^2}$$
for suitable choices of $A,B$.
\medskip

The paper is organized as follows. In section \ref{sec:contained in X} we define the set $\BX$ explicitly and prove Theorem \ref{sat1} and Theorem \ref{sat2}. In section \ref{sec:H-dimension} we reduce the proof of Theorem \ref{prop-H-bound} to a result (Proposition \ref{prop-blablabla}) estimating the Hausdorff dimension of the limit sets of the Fuchsian groups associated to certain infinite degree covers of $\Sigma$. Proposition \ref{prop-blablabla} is proved in section \ref{sec-cheeger}: the main tool for the proof is the work of Patterson, Sullivan, Cheeger and Buser relating Hausdorff dimensions of limit sets to isoperimetric quantities. We conclude in section \ref{sec:comments} with some comments on the sets $X_f$ for linear functions. 

\section{Proof of Theorem \ref{sat1} and Theorem \ref{sat2}}\label{sec:contained in X}

In this section we prove Theorem \ref{sat1} and Theorem \ref{sat2}. Suppose that $\Sigma$ is a closed hyperbolic surface and denote by $\CM\CL(\Sigma)$ the set of measured laminations on $\Sigma$. Let also $\CM\CL_{\min}(\Sigma)\subset\CM\CL(\Sigma)$ be the subset consisting of those measured laminations $\mu$ whose support $\supp(\mu)$ is {\em minimal}, meaning that each half-leaf is dense in $\supp(\mu)$. Abusing terminology we will often identify measured laminations and their  supports. We refer to \cite{Casson-Bleiler, Javi-Cris} for basic facts on laminations and measured laminations. 
   \begin{figure}[ht]
  \setlength{\unitlength}{0.01\linewidth}
  \begin{picture}
  (100, 40)
  \put(10,0){
  \includegraphics{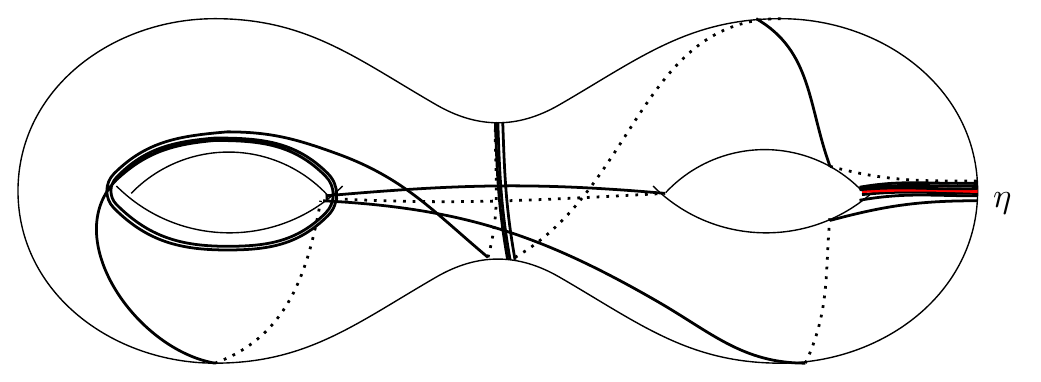}
  }
  \end{picture}
  \caption{A typical element in $\hat\eta$, where $\eta$ is a simple closed curve}
  \label{Fig:hat}
  \end{figure} 
 
Given $\mu\in\CM\CL(\Sigma)$ we denote by $\hat\mu\subset PT\Sigma$ the set of all (lifts of) geodesics in $\Sigma$ which do not meet $\supp(\mu)$ transversally. We are going to prove that the set
\begin{equation}\label{eq-setX}
\BX=\bigcup_{\mu\in\CM\CL_{\min}(\Sigma)}\hat\mu
\end{equation}
satisfies the claim of Theorem \ref{sat2}. We show first that $X_f\subset\BX$ for every sublinear function.

\begin{prop}\label{prop-contained in X}
Let $\BX$ be as in \eqref{eq-setX}. Then we have  $X_f\subset\BX$ for every function $f:\BR_+\to\BR_+$ which satisfies $\limsup_{t\to\infty}\frac{f(t)}t=0$.
\end{prop}

In the course of the proof of Proposition \ref{prop-contained in X} we will have to work with geodesic currents.  Recall that identifying the universal cover of $\Sigma$ with the hyperbolic plane $\BH^2$ we obtain also an identification between the fundamental group $\pi_1(\Sigma)$ and a Fuchsian group $\Gamma\subset\PSL_2\BR$ with $\Sigma$ isometric to $\BH^2/\Gamma$. All those identifications also yield an identification between $PT\Sigma$ and $PT\BH^2/\Gamma$. The advantage of working with $PT\BH^2$ is that while the geodesic foliation of $PT\Sigma$ has dense leaves, that of $PT\BH^2$ has a very simple orbit structure. In fact, since geodesics in $\BH^2$ are determined by their endpoints in the boundary $\BS^1=\D_\infty\BH^2$, we have an identification between the space $\CG(\BH^2)$ of (unoriented) geodesics of $\BH^2$ and $(\BS^1\times\BS^1\setminus\Delta)/\text{flip}$ where the flip exchanges the end-points, or equivalently reverses the orientation of geodesics. 

By definition, a geodesic current on $\Sigma=\BH^2/\Gamma$ is a $\Gamma$-invariant Radon  measure on the space $\CG(\BH^2)$. Every geodesic current $\mu$ induces a Radon  measure $\mu\otimes dt$ on $PT\Sigma$ invariant under both the geodesic flow and the flip. In fact, this map yields an identification between currents and measures invariant under geodesic flow and flip, and this map is a homeomorphism when both spaces are endowed with the respective weak-*-topology.

Seeing primitive periodic orbits of the geodesic flow as invariant measures, we can thus consider individual primitive periodic geodesics as currents. In fact, the currents of the form $c\cdot\gamma$ where $c>0$ and $\gamma$ is a periodic geodesic are dense in the space $\CC(\Sigma)$ of all currents on $\Sigma$. Moreover, the length function $c\cdot\gamma\mapsto c\cdot\ell_\Sigma(\gamma)$ extends continuously to a proper function $\ell_\Sigma:\CC(\Sigma)\to\BR_+$. Similarly, there is a bi-continuous bi-homogenous symmetric map, called the intersection form
$$\I:\CC(\Sigma)\times\CC(\Sigma)\to\BR_+$$
which when evaluated on (the currents associated to) primitive periodic geodesics gives the number of transverse intersections. In particular, simple curves or more generally currents supported by laminations,~i.e.~measured laminations, have vanishing self-intersection number. Conversely, currents with vanishing self-intersection are measured laminations:
$$\CM\CL(\Sigma)=\{\lambda\in\CC(\Sigma)\vert\I(\lambda,\lambda)=0\}$$
We refer to \cite{Bonahon86,Bonahon88,Javi-Cris} for a thorough treatment of geodesic currents.

\begin{proof}[Proof of Proposition \ref{prop-contained in X}]
Recall that $X_f$ is the set of points in $PT\Sigma$ which belong to some Hausdorff limit of a sequence of pairwise distinct $f$-simple geodesics. In particular, the claim follows once we show that 
\begin{quote}
{\em if $(\gamma_n)$ is a sequence of $f$-simple curves with $\ell_\Sigma(\gamma_n)\to\infty$ and converging in the Hausdorff topology to some set $\lambda$, then there is $\mu\in\CM\CL_{\min}(\Sigma)$ with $\lambda\subset\hat\mu$.}
\end{quote}
To find the desired measured lamination $\mu$ consider the currents $\frac 1{\ell_\Sigma(\gamma_n)}\gamma_n$. Since they have unit length, and since the length function is proper on the space of currents $\CC(\Sigma)$, we can pass to a subsequence and assume that they converge to some current $\mu$ when $n\to\infty$:
$$\mu=\lim_{n\to\infty}\frac 1{\ell_\Sigma(\gamma_n)}\gamma_n.$$
Now notice that, by the bi-continuity and bi-homogeneity of the intersection form on the space of currents, we have
\begin{align*}
\I(\mu,\mu)&=\lim_{n\to\infty}\I\left(\frac 1{\ell_\Sigma(\gamma_n)}\gamma_n,\frac 1{\ell_\Sigma(\gamma_n)}\gamma_n\right)=
\lim_{n\to\infty} \frac 1{\ell_\Sigma(\gamma_n)^2}\I(\gamma_n,\gamma_n)\\
&\le\lim_{n\to\infty} \frac 1{\ell_\Sigma(\gamma_n)^2}f(\ell_\Sigma(\gamma))=0.
\end{align*}
This implies that the current $\mu$ is actually a measured lamination. We are going to show next that the Hausdorff limit $\lambda$ of the sequence $(\gamma_n)$ is contained in $\hat\mu$. Note that $\lambda$ is a union of complete geodesics to which we will refer as leaves. To prove that $\lambda\subset\hat\mu$ it suffices to prove that none of the leaves of $\lambda$ intersects the support of $\mu$ transversally. If that is not the case, we can find a geodesic segment $I\subset \lambda$ with endpoints in the complement of $\supp(\mu)$ and with $\I(I,\mu)>0$ . We lift the segment $I$ to the universal cover $\BH^2$ of $\Sigma$ and still denote it by $I$. Let $U\subset\CG(\BH^2)$ be the set of geodesics in $\BH^2$ which meet $I$ and let $V\subset U\setminus\D U$ be a compact subset with 
\begin{equation}\label{abcde}
\mu(V)>0\text{ and }\mu(\D V)=0.
\end{equation}
Note that the condition that $V$ is a compact subset of $U\setminus\D U$ implies that every geodesic in $V$ meets every arc in $\BH^2$ which is sufficiently close to $I$ and hence to suitably chosen arcs $I_n$ contained in lifts of $\gamma_n$. On the other hand we get from \eqref{abcde} and from the fact that the currents $\frac 1{\ell_\Sigma(\gamma_n)}\gamma_n$ converge to $\mu$ with respect to the weak-*-topology that
$$\lim_{n\to\infty}\left(\frac 1{\ell_\Sigma(\gamma_n)}\gamma_n\right)(V)=\mu(V)>0$$
Altogether, this implies that there is $c>0$ such that, for all $n$ large enough, the arc $I_n$ meets at least $c\cdot\ell_\Sigma(\gamma_n)$ lifts of $\gamma_n$. Since $I_n$ was itself contained in a lift of $\gamma_n$ we get that
$$\I(\gamma_n,\gamma_n)\ge c\cdot\ell_{\Sigma}(\gamma_n)$$
which contradicts the assumption that the geodesics $\gamma_n$ are $f$-simple for some sublinear function $f$.
\end{proof}

Having proved that $X_f\subset\BX$ if $f$ is sublinear, we turn now to the other inclusion:


\begin{prop}\label{prop-contains X}
If $f:\BR_+\to\BR_+$ is a monotone unbounded function then we have $\BX\subset X_f$. If moreover $\limsup_{t\to\infty}\frac{f(t)}t=\infty$, then $X_f=PT\Sigma$.
\end{prop}

Before launching the proof of Proposition \ref{prop-contains X} note that Theorem \ref{sat1} follows immediately from the second claim of the proposition:

\begin{named}{Theorem \ref{sat1}}
If $f:\BR_+\to\BR_+$ is such that $\limsup_{t\to\infty}\frac{f(t)}t=\infty$, then $X_f=PT\Sigma$ for every closed hyperbolic surface $\Sigma$.\qed
\end{named}

Similarly, Theorem \ref{sat2} follows when we combine Proposition \ref{prop-contained in X} and the first claim of Proposition \ref{prop-contains X}:

\begin{named}{Theorem \ref{sat2}}
If $\Sigma$ is a closed hyperbolic surface, then there is a set $\BX(\Sigma)\subset PT\Sigma$ with $X_f=\BX(\Sigma)$ whenever $f:\BR_+\to\BR_+$ is unbounded and satisfies $\limsup_{t\to\infty}\frac{f(t)}t=0$.\qed
\end{named}

The remaining of this section is dedicated to proving Proposition \ref{prop-contains X}. Assume  from now on that $f$ is unbounded. Recalling the definition \eqref{eq-setX} of $\BX$, note that to prove the inclusion $\BX\subset X_f$ it suffices to show that
$$\hat\mu\subset X_f$$
for every $\mu\in\CM\CL_{\min}(\Sigma)$. Let us first prove  this for the case when $\mu$ is a simple closed curve:

\begin{lem}\label{Lem:curve hat}
Suppose that $\eta\subset\Sigma$ is a simple closed curve and suppose that $f:\BR_+\to\BR_+$ is unbounded. Then we have $\hat\eta\subset X_f$.
\end{lem}
\begin{proof}
The set $\hat\eta$ can be identified with the set of all bi-infinite geodesics  in the completion $\overline{\Sigma\setminus\eta}$ of the complement of $\eta$ (compare with figure \ref{Fig:Sandwich}). The surface $\overline{\Sigma\setminus\eta}$ is compact, hyperbolic and with totally geodesic boundary. In particular, the periodic orbits of the geodesic flow are dense in the set $\hat\eta$, meaning that every geodesic $\alpha$ in $\hat\eta$ is contained in the Hausdorff limit of a sequence of closed geodesics  disjoint from $\eta$. Note also that, since $X_f$ is by definition a closed subset of $PT\Sigma$, we see that to prove that $\alpha$ is contained in $X_f$ it suffices to show that any closed  geodesic in $\hat\eta$ is contained in $X_f$. 

Let $\gamma\in \hat\eta$ be a closed geodesic. We will prove $\gamma\in X_f$. Fix a base point $p\in\gamma$  and fix a geodesic arc $\omega$ based at $p$ and which meets $\eta$ non-trivially. Fix orientations of $\gamma$ and $\omega$. Choose also a sequence of integers $(a_n)$ whose order of growth we will determine later. And consider the geodesics freely homotopic to the element 
$$\beta_n=D_\eta^{a_n}(\gamma^{2n}*\omega)$$
where $D_\eta$ is the Dehn twist around $\eta$.

\begin{fact}\label{fact1}
For any choice of $(a_n)$ we have that the geodesic $\gamma$ is contained in the Hausdorff limit of the sequence $\beta_n$.
\end{fact}
\begin{proof}
 Let $\tilde\gamma$ be any lift of $\gamma$ to $\mathbb H^2$. It suffices to find a sequence of lifts 
 $\tilde\beta_n$ of $\beta_n$ such that the corresponding endpoints of  $\tilde\beta_n$ converge to those of $\tilde\gamma$. In fact we will
 lift the actual closed path $\beta_n=\eta^{a_n}(\gamma^{2n}*\omega)$, since lifts of homotopic curves have the same endpoints. 
 
   \begin{figure}[ht]
  \setlength{\unitlength}{0.01\linewidth}
  \begin{picture}(100, 55)
  \put(25,0){
  \includegraphics{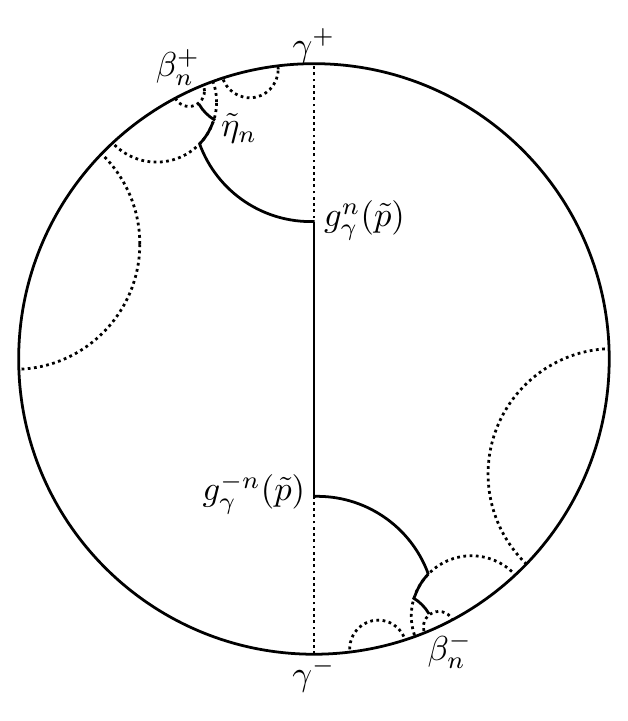}
  }
  \end{picture}
  \caption{A lift of $\beta_n$ to the universal cover.}
  \label{Fig:Sandwich}
  \end{figure}

 Fix a lift $\tilde p\in \tilde \gamma$ of $p$. Denote by $\gamma^+$ and $\gamma^-$ the endpoints of $\tilde\gamma$, and let $g_\gamma$ be the deck transformation
 with axis $\tilde\gamma$, translation length $\ell_\Sigma(\gamma)$ and attracting point $\gamma^+$. 
 Then for any $n\geq 1$ let $\tilde\beta_n$ be the lift
 of $\beta_n$ that contains the segment $[g_\gamma^{n}(\tilde p),g_\gamma^{-n}(\tilde p)]\subset \tilde\gamma$, and denote $\beta^+_n$ and
$ \beta^-_n$ its endpoints at infinity. Let 
 $\delta_n$ be  the connected component of $\tilde\beta_n\setminus [g_\gamma^{n}(\tilde p),g_\gamma^{-n}(\tilde p)]$ 
 with endpoints $g_\gamma^{n}(\tilde p)$ and  $\beta^+_n$. The path $\delta_n$ from $g_\gamma^{n}(\tilde p)$
 goes a bounded distance along a lift of $\omega$, until it reaches  a lift $\tilde\eta_n$ of $\eta$, travels along that geodesic, eventually exiting on the other side of $\tilde\eta_n$ and continuing along another lift of $\omega$ and so on. In particular, $\delta_n$ does not intersect $\tilde\eta_n$ again. Hence
 $\beta^+_n$ is separated from $\beta^+$ by $\tilde\eta_n$. Since $\tilde\eta_{n+1}$ is the image of $\tilde\eta_n$ under $g_\gamma$, the endpoints
 of the geodesic lines $\tilde\eta_n$, and hence the points $\beta^+_n$, converge to $\beta^+$. Using a similar argument we can show that
 the sequence  $\beta^-_n$ converges to $\gamma^-$. This finishes the proof. 
 \end{proof}

In light of Fact \ref{fact1} and of the fact that $X_f$ is closed we see that to prove  $\gamma\subset X_f$ it suffices to prove that $\beta_n\in X_f$ for all large $n$. 

To prove this, note that the self-intersection number of $\beta_n$ is independent of the choice of $(a_n)$, since $D_\eta$ is a homeomorphism and therefore
$$\I(\beta_n,\beta_n)=\I(\gamma^{2n}*\alpha,\gamma^{2n}*\alpha).$$ On the other hand the length of $\beta_n$ is bounded below by a multiple of $a_n$. It follows  that if we choose the numbers $a_n$ growing fast enough, then $\beta_n$ is $f$-simple for $n$ large enough. This completes the proof of Lemma \ref{Lem:curve hat}.
\end{proof}

A similar argument proves the following Lemma.

\begin{lem}\label{lem-I have a headache}
Suppose that $f$ is superlinear. Then every $\gamma\subset\Sigma$ closed is contained in $X_f$.
\end{lem}

Note that we are not claiming that $\gamma$ is $f$-simple, only that it is contained in the Hausdorff limit of a sequence of $f$-simple geodesics.

\begin{proof}
Let $\gamma$ be any closed geodesic in $\Sigma$. As in Lemma \ref{Lem:curve hat} we will construct a sequence of $f$-simple curves $\beta_n$
whose Hausdorff limit contains $\gamma$. Fix a point $p\in\gamma$ so that there is a simple closed geodesic $\alpha$  passing through $p$ that is 
distinct from $\gamma$. Fix orientations for  $\gamma$ and $\alpha$, and let $(a_n)$ be a sequence of integers, rate of growth to be determined later.  Define $\beta_n$ to be a closed curve that, starting at $p$, travels $n$ times about $\gamma$ and then $a_n$ times about $\alpha$, that is
$$\beta_n=\gamma^n * \alpha^{a_n}.$$ 
The argument that $\gamma$ is contained in the Hausdorff limit of $\beta_n$ is essentially the same as in Fact \ref{fact1}. 
We now claim that taking $a_n$  to be large enough one can ensure that $\beta_n$ is $f$-simple. 
Indeed, the self-intersection number of $\beta_n$ satisfies for some $C$ independent of $n$
$$\I(\beta_n,\beta_n)\leq C (\I(\gamma^n,\gamma^n)+a_nn\I(\alpha,\gamma)).$$
On the other hand, the length of $\beta_n$ satisfies
$$\ell_\Sigma(\beta_n)\geq C'a_n$$ for some $C'$ independent of $n$. Given a superlinear function $f$, we can now choose $a_n$ to guarantee
$$\I(\beta_n,\beta_n)\leq f(\ell_\Sigma(\beta_n))$$ for large enough $n$. 
\end{proof}

So far we have only considered the set $\hat\mu$ when $\mu$ is a simple closed curve. To consider the case of a general $\mu\in\CM\CL_{\min}(\Sigma)$ note that, since $\mu$ is minimal, we can approximate it by simple closed curves $\eta_n$ in the Hausdorff topology. We discuss next the relation between the sets $\hat\mu$ and $\hat\eta_n$:

\begin{lem} \label{Lem:hats converge}
Given $\mu\in\CM\CL_{\min}(\Sigma)$, suppose that $\eta_n$ is a sequence of simple closed curves converging in the Hausdorff topology to $\supp(\mu)$ and suppose that the sets $\hat\eta_n$ converge in the Hausdorff topology to some set $\lambda$. Then $\hat\mu\subset\lambda$.
\end{lem}

Actually we have $\lambda=\hat\mu$, but we will only need the given inclusion.

\begin{proof}
To simplify notation we will drop the distinction between the measured lamination $\mu$ and its support $\supp(\mu)$. 

Let $\gamma$ be a geodesic which does not intersect $\mu$ transverselly. To prove  $\gamma\in\lambda$ it suffices  to
find a sequence of geodesics $\gamma_k\in \hat\eta_{n_k}$ for some $n_k\to\infty$, whose Hausdorff limit contains $\gamma$. Noting that we can choose $\gamma_k=\eta_k$ if $\gamma\subset\mu$, we suppose from now on that this is not the case, that is that $\gamma\cap\mu=\emptyset$.

If the distance between $\gamma$ and $\mu$ is positive, then for $n$ large enough $\gamma$ is in $\hat\eta_n$.
Hence we will assume that $\gamma$ comes arbitrarily close to $\mu$. In particular $\gamma$ is not a closed curve.  Both ends of $\gamma$
could get arbitrary close to $\mu$, or only one of them.  We assume the first option, the argument is easily modifiable for the second one.
Fix $p\in \gamma$ and denote $\tau_k$ the segment of $\gamma$ centered at $p$ and of length $k$. Since $\tau_k$ is disjoint from
some neighbourhood of  $\mu$, it is disjoint from all $\eta_n$  for $n$ large enough. Let $\gamma_k$ be a bi-infinite path that contains $\tau_k$, continues along $\gamma$ in both directions until it hits some $\eta_{n_k}$ where $n_k$ is large enough,  then turns on $\eta_{n_k}$ so that the angle it makes is at most $\pi/2$, and wraps around $\eta_{n_k}$ for the rest of its life. 
The geodesic representative of $\gamma_n$ is disjoint from $\eta_{n_k}$, and hence is in $\hat\eta_{n_k}$. 

To see that $\gamma$ is contained in the Hausdorff limit of $\gamma_k$ we argue basically as in the proof of Fact \ref{fact1}. That is, we fix a lift $\tilde p$ of $p$ in the universal cover,
and consider the lifts  $\tilde\gamma$ and $\tilde\gamma_k$ of $\gamma$ and $\gamma_k$ through $\tilde p$. 
 Starting at $\tilde p$ in either direction, $\tilde \gamma_k$ goes along $\tilde\gamma$ distance at least 
$k/2$, and  turns an angle of at most $\pi/2$ to continue along a lift of $\eta_{n_k}$. Hence
the endpoints of $\tilde\gamma_k$ converge to the corresponding endpoints of $\tilde\gamma$. This finishes the proof. 
\end{proof}

Combining the lemmas above  we can prove Proposition \ref{prop-contains X}.

\begin{proof}[Proof of Proposition \ref{prop-contains X}]
Suppose that we are given $\mu\in\CM\CL_{\min}(\Sigma)$ and let $(\eta_n)$ be a sequence of simple closed curves which converge to $\supp(\mu)$ in the Hausdorff topology. Note also that, up to passing to a subsequence, we can assume that the sets $(\hat\eta_n)$ converge in the Hausdorff topology to some set $\lambda$. By Lemma \ref{Lem:curve hat} we have that $\hat\eta_n\subset X_f$ for all $n$. Since $X_f$ is closed we  also have that $\lambda\subset X_f$. Now, Lemma \ref{Lem:hats converge} asserts that $\hat\mu\subset\lambda$ and hence that $\hat\mu\subset X_f$. Since $\mu\in\CM\CL_{\min}(\Sigma)$ was arbitrary, we have proved that
$$\BX=\cup_{\mu\in\CM\CL_{\min}(\Sigma)}\hat\mu\subset X_f.$$
Now, if $f$ is superlinear then we also get from Lemma \ref{lem-I have a headache} that every periodic geodesic belongs to $X_f$. Since the periodic geodesics are dense in $PT\Sigma$ and since $X_f$ is closed we get thus that $PT\Sigma\subset X_f$, as we needed to prove.
\end{proof}

\section{Hausdorff dimension of the set $\BX$}\label{sec:H-dimension}
In this section we will give an estimate of the Hausdorff dimension of the set $\BX=\BX(\Sigma)$ defined in \eqref{eq-setX}. As a first step we observe that it basically will suffice to estimate the Hausdorff dimension of the set
$$X'(\Sigma)=\bigcup_{\eta\in\CS(\Sigma)}\hat\eta,$$
where $\CS(\Sigma)$ is the set of all simple closed geodesics in $\Sigma$. In fact we have:

\begin{lem}\label{lem-simplerX}
For every closed surface $\Sigma$ there is a 3-to-1 cover $\Sigma'\to\Sigma$ with 
$$\dim_\CH X'(\Sigma)\le\dim_\CH\BX(\Sigma)\le \dim_\CH X'(\Sigma')$$
\end{lem}

Before proving Lemma \ref{lem-simplerX} we recall a few facts about the topology of geodesic laminations, see \cite{Casson-Bleiler} for details. Suppose that $\mu\subset\Sigma$ is a geodesic lamination and, abusing notation, denote by $\Sigma\setminus\mu$ both the complement of $\mu$ in $\Sigma$ as well as the completion of this set. The surface $\Sigma\setminus\mu$ is then a hyperbolic surface of finite area with totally geodesic boundary. In particular, it has finitely many components. It might well be that all the components of $\Sigma\setminus\mu$ are ideal polygons, but we will be mainly interested in laminations for which this is not the case. Suppose that that some component $U$ of $\Sigma\setminus\mu$ is not an ideal polygon and let $\Sigma^U\to\Sigma$ be the cover of $\Sigma$ corresponding to the subgroup $\pi_1(U)$ of $\pi_1(\Sigma)$. The surface $\Sigma^U$ has a uniquely determined Nielsen core $N(\Sigma^U)$ - recall that Nielsen core is just another name for convex core. The component $U$ lifts isometrically to $\Sigma^U$ and its lift contains $N(\Sigma^U)$. Every connected component $C$ of $U\setminus N(\Sigma^U)$ is a so-called {\em crown}. Since each crown $C$ retracts to the corresponding boundary component of the Nielsen core $N(\Sigma^U)$, it follows that each geodesic arc in $C$ with enpoints in $C\cap N(\Sigma^U)$ is actually completely contained in $C\cap N(\Sigma^U)$. This has for us the important consequence that
\begin{itemize}
\item[(*)] if $\gamma:\BR\to U$ is a bi-infinite geodesic, then either $\gamma(\BR)\subset\D N(\Sigma^U)$, or $\gamma(\BR)$ meets the boundary $\D N(\Sigma^U)$ of the Nielsen core at most twice.
\end{itemize}
We are now ready to prove the lemma.

\begin{proof}[Proof of Lemma \ref{lem-simplerX}]
Note that the first inequality is obvious because $$X'(\Sigma)\subset\BX(\Sigma).$$ The remaining of the proof is devoted to proving the second inequality.

Given a simple closed geodesic $\eta\in\CS(\Sigma)$, let $\CM_\eta$ be the set of minimal measured laminations $\mu$ with the property that $\eta$ appears as a boundary component of the Nielsen core $N(\Sigma^U)$ of the cover corresponding to some component $U$ of $\Sigma\setminus\mu$ which is not simply connected. Said differently, $\mu\in\CM_\eta$ if $\eta$ bounds a crown of $\mu$. Let also $\CF$ be the set of minimal measured laminations $\mu$ whose complement $\Sigma\setminus\mu$ consists of ideal polygons. We can thus write 
$$\BX=\bigcup_{\mu\in\CM\CL_{\min}(\Sigma)}\hat\mu=\bigcup_{\eta\in\CS(\Sigma)}\left(\bigcup_{\mu\in\CM_\eta}\hat\mu\right)\cup\left(\bigcup_{\mu\in\CF}\hat\mu\right)$$
where the first equality is just the definition of $\BX$ and where the second follows because 
$$\CM\CL_{\min}(\Sigma)=\left(\bigcup_{\eta\in\CS(\Sigma)}\CM_\eta\right)\cup\CF.$$
In particular, since we are dealing with a countable union, we have that the Hausdorff dimension of $\BX$ is equal to the supremum of the Hausdorff dimensions of each one of the sets $\cup_{\mu\in\CM_\eta}\hat\mu$ and $\cup_{\mu\in\CF}\hat\mu$.

Suppose that $\mu\in\CF$. Since each component of $\Sigma\setminus\mu$ is an ideal polygon, it follows that each geodesic contained in $\hat\mu$ is simple. In particular, $\cup_{\mu\in\CF}\hat\mu$ is contained in the set of all simple geodesics and hence has Hausdorff dimension $1$ by the result of Birman-Series after which this note has been named \cite{Birman-Series}. It thus follows that
$$\dim_\CH\BX=\sup_{\eta\in\CS(\Sigma)}\left(\dim_\CH\left(\bigcup_{\mu\in\CM_\eta}\hat\mu\right)\right).$$
Since the Hausdorff dimension does not change under covers and since every closed surface has finitely many covers of degree 3, we obtain the second inequality in Lemma \ref{lem-simplerX} once we prove the following claim:

\begin{fact}\label{fact-cover-31}
For every simple closed geodesic $\eta\in\CS(\Sigma)$ there is a 3-to-1 normal cover $\pi:\Sigma_0\to\Sigma$ such that whenever $\eta_0\in\CS(\Sigma_0)$ is a connected component of $\pi^{-1}(\eta)$, then $\cup_{\mu\in\CM_\eta}\hat\mu\subset\pi(\hat\eta_0)$.
\end{fact}
\begin{proof}
To find an appropriate cover consider a non-separating simple closed curve $\alpha\subset\Sigma$ disjoint from $\eta$ and consider the degree $3$ cover 
\begin{equation}\label{eq-cover}
\pi:\Sigma_0\to\Sigma
\end{equation}
corresponding to the kernel of the composition of the homomorphisms 
$$\pi_1(\Sigma)\to \BZ/3\BZ,\ \ [\gamma]\mapsto\langle\gamma,\alpha\rangle\mod 3$$
where $\langle\cdot,\cdot\rangle$ is the algebraic intersection number. Note that the cover \eqref{eq-cover} is normal and has the property that $\pi^{-1}(\eta)$ consists of 3 homeomorphic lifts $\eta_0,\eta_1,\eta_2$. 

Given $\mu\in\CM_\eta$ we need to show that $\hat\mu\subset \pi(\hat\eta_0)$. Since the cover is normal, it actually suffices to prove that every geodesic $\gamma:\BR\to\Sigma$ with $\gamma(\BR)$ disjoint of $\mu\in\CM_\eta$ has a lift $\tilde\gamma:\BR\to\Sigma_0$ which is disjoint from one of the three curves $\eta_0,\eta_1,\eta_2$. However, since by definition of $\CM_\eta$ we have that $\eta$ bounds a crown in $\Sigma\setminus\mu$, we get from (*) that every geodesic $\gamma:\BR\to\Sigma\setminus\mu$ intersects $\eta$ at most twice. In particular, the lifts of $\gamma$ to $\Sigma_0$ intersect at most two of the three preimages of $\eta$, as we needed to show.
\end{proof}

Having proved Fact \ref{fact-cover-31} we have also proved Lemma \ref{lem-simplerX}.
\end{proof}

In light of Lemma \ref{lem-simplerX}, to estimate $\dim_\CH\BX(\Sigma)$ we only need to be able to estimate the dimension of the sets $\hat\eta$ where $\eta\subset\Sigma$ is a simple closed curve. For any such $\eta\in\CS(\Sigma)$ let $\Sigma_\eta$ be the unique complete hyperbolic surface with Nielsen core 
$$N(\Sigma_\eta)=\Sigma\setminus\eta.$$
As earlier, we are abusing notation and denoting by $\Sigma\setminus\eta$ not only the complement of $\eta$ in $\Sigma$ but also its completion. For every component $\Sigma_\eta^i$ of $\Sigma_\eta$ fix a Fuchsian group $\Gamma_\eta^i$ with $\Sigma_\eta^i$ isometric to $\BH^2/\Gamma_\eta^i$. Finally, let $\Lambda(\Gamma_\eta^i)\subset\BS^1=\D_\infty\BH^2$ be the limit set of $\Gamma_\eta^i$ and, to avoid distinguishing as far as possible if the curve $\eta$ is separating or not, we set $\dim_\CH\Lambda(\Gamma_\eta)$ to be the maximum of $\dim_\CH\Lambda(\Gamma_\eta^i)$ over all components.

\begin{lem}\label{lem-dimdimlim}
With the same notation as above we have
$$\dim_\CH\hat\eta=1+2\dim_\CH\Lambda(\Gamma_\eta)$$
for every simple curve $\eta\in\CS(\Sigma)$.
\end{lem}

\begin{proof}
We prove the lemma for $\eta$ non-separating. The argument in the separating case is almost identical and we leave it to the interested reader. 
   \begin{figure}[ht]
  \setlength{\unitlength}{0.01\linewidth}
  \begin{picture}(90, 60)
  \put(0,0){
  \includegraphics{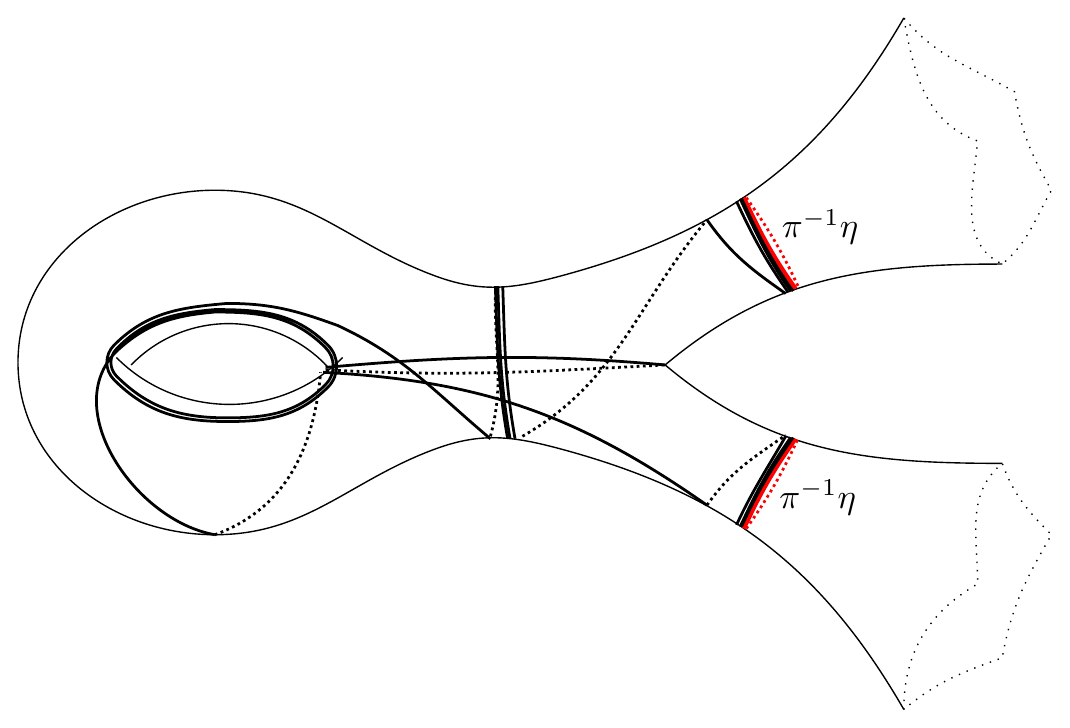}
  }
  \end{picture}
  \caption{The surface $\Sigma_\eta$ for a non-separating simple closed curve $\eta$ - compare with figure \ref{Fig:hat}. The Nielsen core is bounded by the two lifts of $\eta$.}
  \label{Fig:Sandwich}
  \end{figure}

Starting with the proof, recall that $\Sigma_\eta=\BH^2/\Gamma_\eta$ and denote by $\CF(\Sigma_\eta)\subset PT\BH^2$ be the set of all directions tangent to geodesics which spend all their life, past and future, inside the preimage under $\BH^2\to\Sigma_\eta$ of the Nielsen core $N(\Sigma_\eta)$ of $\Sigma_\eta$. Equivalently, $\CF(\Sigma_\eta)$ is the set of lines tangent to geodesics whose two accumulation points in $\D\BH^2$ belong to the limit set $\Lambda(\Gamma_\eta)$. By definition, $\hat\eta\subset PT\Sigma$ is the image of $\CF(\Sigma_\eta)$ under the covering 
$$PT\BH^2\to PT\Sigma_\eta\to PT\Sigma.$$
Noting that coverings preserve Hausdorff dimension, we get that
$$\dim_\CH\hat\eta=\dim_{\CH}\CF(\Sigma_\eta)=1+2\dim_\CH\Lambda(\Gamma_\eta)$$
where the last equality is surely well-known. We explain anyways why it is true. Given a geodesic $\tau\subset\BH^2$ whose end points $\tau(\pm\infty)\in\BS^1=\D_\infty\BH^2$ are disjoint from $\Lambda_\eta$, consider the two intervals $I,J$ into which $\tau(\pm\infty)$ cuts $\BS^1$. Suppose also that $A=I\cap\Lambda(\Gamma_\eta)$ and $B=I\cap\Lambda(\Gamma_\eta)$ are not empty. Then we have a map
$$\BR\times A\times B\to\CF(\Sigma_\eta)$$
which maps $(t,a,b)$ to $\gamma_{a,b}(t)$ where $\gamma_{a,b}$ is the geodesic parametrized by arc length with
$$\gamma_{a,b}(0)\in\tau\ \ \lim_{t\to-\infty}\gamma_{a,b}(t)=a\ \ \lim_{t\to\infty}\gamma_{a,b}(t)=b.$$
Since this map is the restriction of a smooth embedding $\BR\times I\times J\to PT\BH^2$ we deduce that the Hausdorff dimension of its image is equal to
$$1+\dim_\CH A+\dim_\CH B=1+2\dim_\CH\Lambda(\Gamma_\eta)$$
because the limit set is self-similar. Since $\CF(\Sigma_\eta)$ can be covered by countably many sets of this form, the claim follows.
\end{proof}

In the next section we will prove the following estimate of the maximum of the Hausdorff dimensions of limit sets of the Fuchsian groups $\Gamma_\eta$ where $\eta$ ranges over all simple closed geodesics:

\begin{prop}\label{prop-blablabla}
Let $\Sigma$ be a closed hyperbolic surface with systole $\syst(\Sigma)$, and denote by $\CS(\Sigma)$ the set of all simple closed geodesics in $\Sigma$. Then we have
$$\max_{\eta\in\CS(\Sigma)}\dim_\CH\Lambda(\Gamma_\eta)\le\frac 12+\frac 12\sqrt{1-\frac{4\cdot \syst(\Sigma)^2}{\vol(\Sigma)^2+4\cdot\syst(\Sigma)^2}}.$$
If moreover $10\cdot\syst(\Sigma)\le\vol(\Sigma)$ then we also have
$$\frac 12+\frac 12\sqrt{1-10\frac{\syst(\Sigma)}{\vol(\Sigma)}}\le\max_{\eta\in\CS(\Sigma)}\dim_\CH\Lambda(\Gamma_\eta).$$
\end{prop}

Before moving on recall that the systole can be bounded from above in terms of the genus $g$ as follows: since $\vol(\Sigma)=2\pi(2g-2)$ by Gau\ss -Bonnet, and since the volume of a hyperbolic ball of radius $r$ is given by $2\pi(\cosh r-1)$, it follows that if $r$ is the volume of the largest embedded ball then $r\le\arccosh(2g-1)$ which means that 
\begin{equation}\label{eq-silly-bound}
\syst(\Sigma)\le 2r\le 2\cdot\arccosh(2g-1).
\end{equation}
Now, since the function 
$$x\mapsto\frac{20\arccosh(2x-1)}{4\pi(x-1)}$$
is monotonically decreasing when $x>1$, and since it takes the value $0.983...$ for $x=6$, we deduce that 
\begin{equation}\label{eq-silly condition}
10\cdot\syst(\Sigma)\le\vol(\Sigma)
\end{equation}
as long as $\Sigma$ has genus $6$: 

\begin{lem}\label{kor-syst-bound}
If $\Sigma$ has genus $g\ge 6$ then it satisfies \eqref{eq-silly condition}.\qed
\end{lem}

A similar argument shows that 
\begin{equation}\label{eq-silly bound}
\frac{10}3\cdot\syst(\Sigma)\le\vol(\Sigma)
\end{equation}
for every closed hyperbolic surface $\Sigma$.

Proposition \ref{prop-blablabla} will be proved in the next  section. At this point we use it to conclude the proof of Theorem \ref{prop-H-bound}:

\begin{named}{Theorem \ref{prop-H-bound}}
There are constants $C>0$ and $c>0$ with 
$$3-c\cdot\frac{\syst(\Sigma)}{\vol(\Sigma)}\le  \dim_{\CH}\BX(\Sigma)\le 3-C\cdot\left(\frac{\syst(\Sigma)}{\vol(\Sigma)}\right)^2$$
for every closed hyperbolic surface $\Sigma$. Here $\syst(\Sigma)$ is the systole of $\Sigma$, i.e. the length of the shortest closed geodesic.
\end{named}
\begin{proof}
First let us consider the surfaces $\Sigma$ that satisfy \eqref{eq-silly condition}. Note  that whenever we have a finite cover $\Sigma'\to\Sigma$ then
\begin{equation}\label{eq-I am sick of this}
\frac{\syst(\Sigma')}{\vol(\Sigma')}\le\frac{\syst(\Sigma)}{\vol(\Sigma)}.
\end{equation}
Hence if $\Sigma$ satisfies \eqref{eq-silly condition} then the surface $\Sigma'$ provided by Lemma \ref{eq-silly condition} also satisfies \eqref{eq-silly condition}. Recall that, by the said lemma, $\Sigma'$ has the property that 
$$\dim_\CH X'(\Sigma)\le\dim_\CH\BX(\Sigma)\le\dim_\CH X'(\Sigma').$$
Now, from Lemma \ref{lem-dimdimlim} and the lower bound in Proposition \ref{prop-blablabla} we get that
\begin{align*}
\dim_\CH X'(\Sigma)
    &=\sup_{\eta\in\CS(\Sigma)}\dim_\CH\hat\eta=1+2\sup_{\eta\in\CS(\Sigma)}\dim_\CH\Lambda(\Gamma_\eta)\\
    &\ge 2+\sqrt{1-10\frac{\syst(\Sigma)}{\vol(\Sigma)}}.
\end{align*}    
Similarly, using the upper bound in Proposition \ref{prop-blablabla} we get
$$\dim_\CH X'(\Sigma')\le 2+\sqrt{1-\frac{4\cdot\syst(\Sigma')^2}{\vol(\Sigma')^2+4\cdot\syst(\Sigma')^2}},$$
which in light of \eqref{eq-silly bound} yields also
$$\dim_\CH X'(\Sigma')\le 2+\sqrt{1-\frac{\syst(\Sigma')^2}{12\cdot\vol(\Sigma')^2}}.$$
Now, taking into account \eqref{eq-I am sick of this} we get
$$\dim_\CH X'(\Sigma')\le 2+\sqrt{1-\frac{\syst(\Sigma)^2}{12\cdot\vol(\Sigma)^2}}.$$
Altogether we obtain
$$  2+\sqrt{1-10\frac{\syst(\Sigma)}{\vol(\Sigma)}}\le\dim_{\CH}\BX(\Sigma)\le 2+\sqrt{1-\frac{\syst(\Sigma)^2}{12\cdot\vol(\Sigma)^2}}$$
from where we obtain the desired bound when considering the Taylor expansion of the square root function at $1$. We have proved Theorem \ref{prop-H-bound} for all surfaces $\Sigma$ satisfying \eqref{eq-silly condition}.

Consider now the set $Z$ of all hyperbolic surfaces $\Sigma$ which do not satisfy \eqref{eq-silly condition}. From Lemma \ref{kor-syst-bound} we get that $Z$ is a subset of the union of four moduli spaces $\CM_g$ with $g=2,3,4,5$. In fact, the (closure of the) part of $Z$ contained in each one of those moduli spaces is compact because it is contained in the compact set consisting of surfaces where the injectivity radius is at least $\frac {\vol(\Sigma)}{20}=\frac{\pi(g-1)}5$. Now, compactness of $Z$ implies that the continuous function $\Sigma\mapsto\dim_\CH\BX(\CH)$ achieves a maximum, say at $\Sigma_0$. Given that 
$$\max_{\eta\in\CS(\Sigma)}\dim_\CH\Lambda(\Gamma_\eta)\le\frac 12+\frac 12\sqrt{1-\frac{4\cdot\syst(\Sigma)^2}{\vol(\Sigma)^2+4\cdot\syst(\Sigma)^2}}<1$$
by Proposition \ref{prop-blablabla}, we deduce from Lemma \ref{lem-dimdimlim} that $\dim_\CH\BX(\Sigma_0)<3$. Now, since the function $\frac{\syst(\cdot)}{\vol(\cdot)}$ is also continuous and hence bounded on the compact set $Z$ we get that there are $C,c>0$ with 
$$3-C\cdot\left(\frac{\syst(\Sigma)}{\vol(\Sigma)}\right)^2\ge \dim_{\CH}\BX(\Sigma)\ge 3-c\cdot\frac{\syst(\Sigma)}{\vol(\Sigma)}$$
for every $\Sigma\in Z$. In other words, we have also proved Theorem \ref{prop-H-bound} for all surfaces $\Sigma$ which do not satisfy \eqref{eq-silly condition}. 
\end{proof}

\section{Hausdorff dimension of limit sets}\label{sec-cheeger}

Our next goal is to prove Proposition \ref{prop-blablabla}. In order to estimate the Haudorff dimension $\dim_\CH\Lambda(\Gamma_\eta)$ of the limit set of the Fuchsian group $\Gamma_\eta$ we will rely on relation between this quantity, the bottom of the spectrum of the Laplacian $\lambda_0(\Sigma_\eta)$ of $\Sigma_\eta=\BH^2/\Gamma_\eta$, and the Cheeger constant. To begin with, recall that by the work of Patterson \cite{Patterson} and Sullivan \cite{Sullivan} we have 
$$\lambda_0(\Sigma_\eta)=\dim_\CH\Lambda(\Gamma_\eta)(1-\dim_\CH\Lambda(\Gamma_\eta))$$
whenever $\dim_\CH\Lambda(\Gamma_\eta)\ge\frac 12$ and $\lambda_0(\Sigma_\eta)=\frac 14$ otherwise. Since we want to compute $\dim_\CH\Lambda(\Gamma_\eta)$ we can rewrite as
$$\dim_\CH\Lambda(\Gamma_\eta)=\frac 12+\sqrt{\frac 14-\lambda_0(\Sigma_\eta)}$$
if $\lambda_0(\Sigma_\eta)<\frac 14$ and $\dim_\CH\Lambda(\Gamma_\eta)\le\frac 12$ otherwise. 

It thus follows that, in order to bound $\dim_\CH\Lambda(\Gamma_\eta)$ from below we have to bound $\lambda_0(\Sigma_\eta)$ from above. To do so, recall that almost by definition 
$$\lambda_0(\Sigma_\eta)=\inf\{\CR(f)\vert f\in C_c(\Sigma_\eta)\}$$ 
is the infimum of the Rayleigh quotients $R(f)=\frac{\int\Vert\DD f\Vert^2}{\int f^2}$ over all functions with compact support.

\begin{lem}\label{H-lowerbound}
Let $\Sigma$ be a closed hyperbolic surface of genus $g$ and suppose that $\eta\subset\Sigma$ is a simple closed geodesic with 
$$10\cdot \ell_\Sigma(\eta)\le\vol(\Sigma).$$ 
Then we have
$$\dim_{\CH}\Lambda(\Gamma_\eta)\ge\frac 12+\frac 12\sqrt{1-10\frac{\ell_\Sigma(\eta)}{\vol(\Sigma)}}.$$
\end{lem}
\begin{proof}
Consider the function
$$f:\Sigma_\eta\to[0,1],\ \ f(x)=\max\{0,1-d(x,N(\Sigma_\eta))\}.$$
The gradient $\DD f$ of this function vanishes outside of the collar 
$$C=\{x\in\Sigma_g\vert 0< d(x,N(\Sigma_\eta))\le 1\},$$
and has norm $1$ there in. Also, the function $f\equiv 1$ on the Nielsen core $N(\Sigma_\eta)$. Hence we can easily estimate the Rayleigh quotient $\CR(f)$ of $f$ as follows:
\begin{align*}
\CR(f)&=\frac{\int\Vert\DD f\Vert^2}{\int f^2}\le\frac{\vol(C)}{\vol(N(\Sigma_\eta))}\\
&=\frac{\sinh(1)\cdot\ell(\D N(\Sigma_\eta))}{\vol(N(\Sigma_\eta))}=\frac{2\sinh(1)\cdot\ell(\eta)}{\vol(\Sigma)}\\
&<\frac 52\cdot \frac{\ell(\eta)}{\vol(\Sigma)}.
\end{align*}
In particular, we get the following upper bound for the bottom of the spectrum of $\Sigma_\eta$:
$$\lambda_0(\Sigma_\eta)\le\CR(f)<\frac 52\cdot\frac{\ell(\eta)}{\vol(\Sigma)}$$
The assumption on the length of the curve implies thus that $\lambda_0(\Sigma_\eta)<\frac 14$, which means that the Hausdorff dimension $\dim_\CH\Lambda(\Gamma_\eta)$ of the limit set of the Fuchian group associated to $\Sigma_\eta$ is given by
$$\dim_\CH\Lambda(\Gamma_\eta)=\frac 12+\sqrt{\frac 14-\lambda_0(\Sigma_\eta)}\ge \frac 12+\sqrt{\frac 14-\frac 52\cdot\frac{\ell(\eta)}{\vol(\Sigma)}},$$
as claimed. 
\end{proof}

We come now to the upper bound in Proposition \ref{prop-blablabla}. In other words, we need to bound from above the Hausdorff dimenson of the limit set $\Lambda(\Gamma_\eta)$ where $\eta$ ranges over the collection $\CS$ of all simple closed geodesics. Since we are going to again exploit the relation between Hausdorff dimension and bottom of the Laplacian, what we need to do is to obtain lower bounds for $\lambda_0(\Sigma_\eta)$ over all choices of $\eta$. The key to do so is the relation between $\lambda_0(\Sigma_\eta)$ and the Cheeger constant:
\begin{equation}\label{eq-cheeger}
h(\Sigma_\eta)=\inf\frac{\ell(\D A)}{\vol(A)}
\end{equation}
where the infimum is taken over all compact subsurfaces $A$ of the infinite volume surface $\Sigma_\eta$. More concretely, Cheeger's inequality \cite{Cheeger} asserts that
$$\frac 14h(\Sigma_\eta)^2\le\lambda_0(\Sigma_\eta).$$
Since we are going to be interested in uniform bounds for $\lambda_0(\Sigma_\eta)$ over all choices of $\eta$ we will be  interested in the quantity
\begin{equation}\label{min-cheeger}
H(\Sigma)=\inf_{\eta\in\CS(\Sigma)}h(\Sigma_\eta).
\end{equation}
We will see that $H(\Sigma)$ satisfies the following inequality.

\begin{lem}\label{lem-funny-cheeger}
We have 
$$H(\Sigma)\ge \frac{2\cdot\syst(\Sigma)}{\sqrt{\vol(\Sigma)^2+4\cdot\syst(\Sigma)^2}}$$
for every closed hyperbolic surface. 
\end{lem}

Assuming Lemma \ref{lem-funny-cheeger} we conclude the proof of Proposition \ref{prop-blablabla}:

\begin{proof}[Proof of Proposition \ref{prop-blablabla}]
We first prove the upper bound. Note that there is nothing to prove if $\dim_\CH\Lambda(\Gamma_\eta)\le\frac 12$ for all $\eta$. 
Otherwise we have
\begin{align*}
\sup_{\eta\in\CS(\Sigma)}\dim_\CH(\Lambda(\Gamma_\eta))
&=\sup_{\eta\in\CS(\Sigma)\text{ with }\dim_\CH\Lambda(\Gamma_\eta)>\frac 12}\dim_\CH(\Lambda(\Gamma_\eta))\\
&=\frac 12+\sqrt{\frac 14-\min_{\eta\in\CS(\Sigma)\text{ with }\lambda_0(\Sigma_\eta)<\frac 14}\lambda_0(\Sigma_\eta)}\\
&\le\frac 12+\frac 12\sqrt{1-\min_{\eta\in\CS(\Sigma)\text{ with }h(\Sigma_\eta)<1}h(\Sigma_\eta)^2}\\
&\le\frac 12+\frac 12\sqrt{1-H(\Sigma)^2},
\end{align*}
where the first inequality holds by the Cheeger's inequality and the second is just the definition of $H(\Sigma)$. Now, by Lemma \ref{lem-funny-cheeger} we have 
$$H(\Sigma)\ge \frac{2\cdot\syst(\Sigma)}{\sqrt{\vol(\Sigma)^2+4\cdot\syst(\Sigma)^2}}.$$
Combining the two inequalities above, we obtain the deisred upper bound
$$\sup_{\eta\in\CS(\Sigma)}\dim_\CH(\Lambda(\Gamma_\eta))\le\frac 12+\frac 12\sqrt{1-\frac{4\cdot\syst(\Sigma)^2}{\vol(\Sigma)^2+4\cdot\syst(\Sigma)^2}}.$$
Now, to prove the lower bound recall that by Lemma \ref{H-lowerbound} we have 
\begin{equation}\label{lazy cat}
\dim_{\CH}\Lambda(\Gamma_\eta)\ge 2+\sqrt{1-10\frac{\ell_\Sigma(\eta)}{\vol(\Sigma)}}
\end{equation}
for every $\eta\in\CS(\Sigma)$ with $10\cdot \ell_\Sigma(\eta)\le\vol(\Sigma)$. By assumption there is such a curve $\eta$ and the desired bound follows from applying \eqref{lazy cat} to any $\eta$ with $\ell_\Sigma(\eta)=\syst(\Sigma)$. This finishes the proof of the proposition.
\end{proof}

The rest of this section is devoted to proving Lemma \ref{lem-funny-cheeger}. We will rely heavily on the fact that the Cheeger constant is realized by a subsurface whose boundary has constant geodesic curvature. This was for instance observed by Buser \cite{Buser}, but the particular case of geometrically finite surfaces  has been discussed in all details in \cite{Benson}.

\begin{proof}[Proof of Lemma \ref{lem-funny-cheeger}]
To begin with,  note that the statement trivially holds if $H(\Sigma)\ge 1$. Note also that if $H(\Sigma)<1$, then to compute $H(\Sigma)$ we only need to consider  curves $\eta$ with $h(\Sigma_\eta)<1$. Therefore suppose that we have such an $\eta$ with $h(\Sigma_\eta)<1$. As we mentioned before the proof, to compute $h(\Sigma_\eta)$ it suffices (see \cite{Benson}) to take the infimum in \eqref{eq-cheeger} over all subsurfaces $A$ satisfying the following property:
\begin{itemize}
\item[(*)] $A$ is a convex subsurface whose boundary has constant geodesic curvature. 
\end{itemize}
Moreover, since $h(\Sigma_\eta)<1$, we can assume that $A$ is neither a disk nor an annulus. 

Given now a candidate $A\subset\Sigma_\eta$ satisfying (*), consider the cover $\Sigma_\eta^A\to\Sigma_\eta$ corresponding to the subgroup $\pi_1(A)$ of $\pi_1(\Sigma_\eta)$, and note that $A$ lifts homeomorphically to $\Sigma_\eta^A$. This means that when computing $h(\Sigma_\eta^A)$ we can also use the candidate $A$. This means that, always under the assumption that $h(\Sigma_\eta)<1$, we can bound the Cheeger constant from below as follows:
\begin{equation}\label{eq-cheeger-bound}
h(\Sigma_\eta)\ge\inf_{A\subset\Sigma_\eta\text{ satisfying (*) and }\chi(A)<0}h(\Sigma_\eta^A)
\end{equation}
Continuing with the same notation, suppose that $B\subset\Sigma_\eta^A$ is any candidate to compute $h(\Sigma_\eta^A)$. Then, taking into account that all covers are geometric in the sense that they correspond to embedded subsurfaces, we obtain that there is a subsurface $B'\subset\Sigma_\eta$ whose corresponding cover $\Sigma_\eta^{B'}$ agrees with the cover of $\Sigma_\eta^A$ corresponding to $B$. Now, since any tower of geometric covers has to stabilize, we deduce that in \eqref{eq-cheeger-bound} we can restrict ourselves to those surfaces $\Sigma_\eta^A$ with the property that $h(\Sigma_\eta^A)$ can be computed using only subsurfaces $B\subset\Sigma_\eta^A$ satisfying (*) and such that the inclusion $B\to\Sigma_\eta^A$ is a homotopy equivalence:
$$h(\Sigma_\eta)\ge\inf\left\{
\frac{\ell(\D B)}{\vol(B)}\middle\vert 
\begin{array}{l}
A\subset\Sigma_\eta\text{ }\pi_1\text{-injective with }\chi(A)<0,\\
B\subset\Sigma_\eta^A\text{ satisfying (*), and such that}\\
B\to\Sigma_\eta^A\text{ is a homotopy equivalence}
\end{array}
\right\}$$
Combining together this last bound for $h(\Sigma_\eta)$ and  the definition of $H$ given in \eqref{min-cheeger} we get  that
\begin{equation}\label{I am sick of this}
H(\Sigma)\ge\inf\left\{
\frac{\ell(\D B)}{\vol(B)}\middle\vert 
\begin{array}{l}
A\subsetneq\Sigma\text{ }\pi_1\text{-injective with }\chi(A)<0,\\
B\subset\Sigma^A\text{ satisfying (*), and such that}\\
B\to\Sigma^A\text{ is a homotopy equivalence}
\end{array}
\right\}
\end{equation}
The advantage of \eqref{I am sick of this} is that for each geometric cover $\Sigma^A\to\Sigma$ it is possible to compute the optimal ratio of $\frac{\ell(\D B)}{\vol(B)}$ over all $B\subset\Sigma^A$ satsifying (*), and such that $B\to\Sigma^A$ is a homotopy equivalence:

\begin{fact}\label{fact-calculus}
Suppose that we are given a geometric cover $\Sigma^A\to\Sigma$ with $\chi(\Sigma^A)<0$. Then we have
$$\inf\left\{
\frac{\ell(\D B)}{\vol(B)}\middle\vert 
\begin{array}{l}
B\subset\Sigma^A\text{ satsifying (*), and such that}\\
B\to\Sigma^A\text{ is a homotopy equivalence}
\end{array}
\right\}=\frac L{\sqrt{V^2+L^2}}$$
where $L=\ell(\D N(\Sigma^A))$ and $V=\vol(N(\Sigma^A))$.
\end{fact}
\begin{proof}
The compact subsurfaces $B$ of $\Sigma^A$ satisfying (*) and such that the inclusion $B\to\Sigma^A$ is a homotopy equivalence are all of the form
$$B(r)=\{x\in\Sigma^A\vert d(x,CC(\Sigma^A)\le r\}\text{ with }r\in[0,\infty).$$
We can compute the volume and the length of the boundary of $B(r)$ explicitly:
$$\vol(B(r))=\vol(N(\Sigma^A))+\sinh(r)\cdot\ell(\D N(\Sigma^A))$$
$$\ell(\D B(r))=\cosh(r)\cdot\ell(\D N(\Sigma^A))$$
Now it is a calculus problem to check that the function
$$r\to\frac{\cosh(r)\cdot\ell(\D N(\Sigma^A))}{\vol(N(\Sigma^A))+\sinh(r)\cdot\ell(\D N(\Sigma^A))}=\frac{\cosh(r)\cdot L}{V+\sinh(r)\cdot L}$$
is minimized if $\sinh(r)=\frac LV$. At this point it takes the value $\frac L{\sqrt{V^2+L^2}}$. We have proved Fact \ref{fact-calculus}.
\end{proof}

Suppose now that $\Sigma^A\to\Sigma$ is a geometric cover and let $\eta\subset\D N(\Sigma^A)$ be a shortest boundary component. Since the cover is geometric, the curve $\eta$ projects homeomorphically to a curve $\eta\subset\Sigma$ of the same length. As always, suppose for the sake of concreteness that $\eta$ is non-separating. Then we have that 
\begin{align*}
&\vol(N(\Sigma^A))\le\vol(\Sigma)=\vol(N(\Sigma_\eta))&\text{   and}\\
&\ell(\D N(\Sigma^A))\ge\ell(\D N(\Sigma_\eta))=2\cdot\ell(\eta)&
\end{align*}
Combining these two inequalities we obtain
\begin{equation}\label{I am cold}
\frac{\ell(\D N(\Sigma^A))}{\sqrt{\vol(N(\Sigma^A))^2+\ell(\D N(\Sigma^A))^2}}\ge\frac{2\cdot \ell(\eta)}{\sqrt{\vol(\Sigma)^2+4\cdot \ell(\eta)^2}}.
\end{equation}
We leave to the reader to check that the bound \eqref{I am cold} also holds if $\eta$ is separating.

Combining now Fact \ref{fact-calculus}, inequality \eqref{I am cold} and bound \eqref{I am sick of this} for $H(\Sigma)$ we get 
$$H(\Sigma)\ge \frac{2\cdot\syst(\Sigma)}{\sqrt{\vol(\Sigma)^2+4\cdot \syst(\Sigma)^2}},$$
which is the bound we wanted to obtain.
\end{proof}

\section{The linear case}\label{sec:comments}

Theorem \ref{sat1} and Theorem \ref{sat2}, together with the Birman-Series theorem \cite{Birman-Series} describe the set $X_f$ for functions with are either superliner or sublinear. The case of linear functions seems different because $X_f$ depends on the concrete function. To make this statement precise consider for $\tau\in(0,\infty)$ the function $x\mapsto\tau x$ and the corresponding set $X^\tau=X_{x\mapsto\tau x}$ and consider also
$$X^0=\cap_\tau X^\tau\text{ and }X^\infty=\overline{\cup_\tau X^\tau}.$$
Note that all these sets are closed and that by definition we have 
$$X^0\subset X^\tau\subset X^{\tau'}\subset X^\infty$$ 
for all $\tau\le\tau'$. The argument used to prove Proposition \ref{prop-contained in X} shows that $X^0\subset\BX$ and hence that $X^0=\BX$ by Proposition \ref{prop-contains X}. Similarly, the argument used in Lemma \ref{lem-I have a headache} proves that  every periodic geodesic $\gamma$ in $\Sigma$ is in a Hausdorff  limit of a sequence of (infinite) geodesics 
$\gamma_n\in X^{\tau_n}$, and therefore  $\gamma\subset X^\infty$. As in the proof of Proposition \ref{prop-contains X}, this implies that $X^\infty=PT\Sigma$.

\begin{lem}\label{lem-linear}
Suppose that $\Sigma$ is a closed surface. We have $X^0=\BX$ and $X^\infty=PT\Sigma$.\qed
\end{lem}

Clearly, Lemma \ref{lem-linear} shows that for every closed hyperbolic surface $\Sigma$ there are $\tau$ and $\tau'$ with $X^\tau\neq X^{\tau'}$.

What we do not know is how  the function
\begin{equation}\label{eq-bla}
[0,\infty]\to\BR_+,\ \tau\mapsto\dim_\CH X^\tau
\end{equation}
behaves. Is it continuous? If not, where are the discontinuities and what is the image of \eqref{eq-bla}?

\end{document}